\documentclass{amsart}

\usepackage{amsmath,amssymb,amscd,amsthm,graphicx,enumerate}
\usepackage{hyperref}
\usepackage{color}
\hypersetup{breaklinks=true}
\usepackage[shortlabels]{enumitem}

\numberwithin{equation}{section}
\theoremstyle{plain}

\newtheorem{theorem}{Theorem}[section]
\newtheorem{proposition}[theorem]{Proposition}
\newtheorem{lemma}[theorem]{Lemma}

\newtheorem{claim}[theorem]{Claim}

\theoremstyle{remark}
\newtheorem{remark}[theorem]{Remark}
\newtheorem{conjecture}[theorem]{Conjecutre}

\theoremstyle{definition}
\newtheorem{definition}[theorem]{Definition}

\newcommand{\eps}{\varepsilon}

\begin{document}

\title[Wulff  inequality for minimal submanifolds in Euclidean space]{Wulff  inequality for minimal submanifolds in Euclidean space}

\author{Wenkui Du}
\address{Department of Mathematics, Massachusetts Institute of Technology}
\email{\tt duwenkui@mit.edu}
\author{Yuchao Yi}
\address{Department of Mathematics, University of California San Diego}
\email{\tt yuyi@ucsd.edu}
\author{Ziyi Zhao}
\address{Institute for Theoretical Sciences, Westlake University }
\email{\tt zhaoziyi@westlake.edu.cn}

\begin{abstract}
In this paper, we prove a  Wulff inequality for $n$-dimensional minimal submanifolds with boundary in $\mathbb{R}^{n+m}$, where we associate a nonnegative anisotropic weight $\Phi: S^{n+m-1}\to \mathbb{R}^{+}$ to the boundary of minimal submanifolds. The Wulff inequality constant depends only on $m$ and $n$, and is independent of the weights. The inequality is sharp if $m=1, 2$ and $\Phi$ is the support function of ellipsoids or certain type of centrally symmetric long convex bodies.
\end{abstract}
\maketitle

\section{Introduction}
The isoperimetric inequality is a fundamental geometric inequality that has been extensively studied throughout history since the ancient era of Queen Dido. It states that the round ball minimizes the boundary hypersurface area among all shapes with a given volume. There are several proofs of the isoperimetric inequality and readers can refer to the books and papers \cite{burago2013geometric, osserman1978isoperimetric, chavel2001isoperimetric}. Later, the isoperimetric inequality for general minimal submanifolds with nonzero codimension in Euclidean space was considered. It was conjectured in \cite{alexander1974area} that
\begin{conjecture}
    For any $n$ dimensional smooth minimal submanifold $\Sigma\subset \mathbb{R}^{n+m}$ with smooth boundary $\partial\Sigma$, the following sharp isoperimetric inequality holds
\begin{align}
    |\partial \Sigma|\geq n|B^n|^{\frac{1}{n}}|\Sigma|^{\frac{n-1}{n}},
\end{align}
and equality holds if and only if $\Sigma$ is an $n$-dimensional ball in $\mathbb{R}^n$. Here, $|\Sigma|$ denotes the $n$-dimensional volume (area) of $\Sigma$, $|\partial \Sigma|$ denotes the $(n-1)$-dimensional volume (area) of the boundary $\partial \Sigma$, $B^n$ is the open unit ball in $\mathbb{R}^n$, and $|B^n|$ denotes its volume.
\end{conjecture}
Carleman \cite{Carleman1921} proved that the conjecture holds for $2$-dimensional minimal surfaces  diffeomorphic to a disk  in $\mathbb{R}^3$ via Wirtinger inequality and Almgren \cite{almgren1986optimal} showed that the conjecture holds for arbitrary codimensional area minimizing submanifolds in Euclidean space. In a recent breakthrough, Brendle \cite{Brendle2021} verified the sharp isoperimetric inequality conjecture for minimal submanifolds in the cases of codimension one and two, adapting ideas from optimal mass transportation \cite{McCannGuillen2013} and ABP method of Cabre \cite{cabre2008elliptic}.

A natural generalization of the isoperimetric inequality considered by Wulff \cite{wul1901frage} is the isoperimetric problem with weighted boundary density.
\begin{definition}
    Let $\Phi$ be a positively one-homogeneous convex function in $\mathbb{R}^n$, and we define the $\Phi$-anisotropic perimeter
    \begin{equation}
        P_{\Phi}(\Omega)=\int_{\partial \Omega} \Phi(\nu(x))dS.
    \end{equation}
    The set
    \begin{equation}
        \mathcal{W} = \{x\in\mathbb{R}^n: \forall \nu \in S^{n-1},  x\cdot\nu \leq \Phi(\nu) \}
    \end{equation}
     is called the corresponding Wulff shape. The support function of $\mathcal{W}$ is as follows:\footnote{It is well-known  that support function and the Wulff shape are mutually determined and the Wulff shape induced by the support function of convex body is precisely the original convex body \cite{taylor1978crystalline}.}
    \begin{equation}
        \Psi(y) = \sup\{x\cdot y: x\in \mathcal{W}\}.
    \end{equation}
\end{definition}

Wulff considered the following question: Given a positive function $\Phi$, what shapes minimize the 
$\Phi$-anisotropic perimeter among the sets of finite perimeter  $\Omega\subset \mathbb{R}^n$ with fixed volume. Wulff
conjectured that the corresponding minimizers are homothetic to the Wulff shape $\mathcal{W}_{\Phi}$ generated by $\Phi$ and the isoperimetric inequality can be extended to the following celebrated Wulff inequality.
\begin{theorem}[Wulff Theorem {\cite{wul1901frage}}]\label{wulff-inequ}
    Let $\mathcal{W}$ be an $n$-dimensional centrally symmetric convex body in $\mathbb{R}^{n}$, and $\Phi$ be the corresponding support function. Then for any set $\Omega \subset \mathbb{R}^n$ of finite perimeter with $|\Omega| < \infty$, we have 
    \begin{equation}
        \frac{P_{\Phi}({\Omega})}{|\Omega|^{\frac{n-1}{n}}} \geq \frac{P_{\Phi}({\mathcal{W}})}{|\mathcal{W}|^{\frac{n-1}{n}}}.
    \end{equation}
    Moreover, equality holds if and only if $\Omega$ is homothetic to $\mathcal{W}$, namely $\Omega = a\mathcal{W}+b$ for some $a>0$ and $b\in \mathbb{R}^n$ up to a set of measure zero.
\end{theorem}

This result was first stated without proof by Wulff in 1901 \cite{wul1901frage}. A complete proof of Theorem  \ref{wulff-inequ}  can be found in Taylor's articles \cite{taylor1974existence, taylor1975unique, taylor1978crystalline}. Cabr{\'e}, Ros-Oton and Serra in \cite{cabre2016sharp} gave a new proof of Theorem \ref{wulff-inequ} via the ABP method.  Figalli, Maggi and Pratelli \cite{figalli2010mass} studied the quantitative version of the codimension zero anisotropic isoperimetric inequality. De Rosa, Kolasiński and Santilli \cite{de2020uniqueness} considered the uniqueness of critical points of the codimensional zero anisotropic isoperimetric problem and  established Heintze-Karcher type inequality. 

A natural question is whether the Wulff inequality holds for minimal submanifolds with nonzero codimension in Euclidean space \footnote{Recently, De Philippis and Pigati \cite{Depp} considered the  Michael-Simon inequality for nonzero codimensional anisotropic minimal submanifolds in Euclidean space via ideas from multilinear Kakeya estimates, where the weight is put on the normal vectors of the submanifolds instead of relative normal vectors on the boundary of  submanifolds.}. To state the inequality, for each $n$-dimensional affine subspace $P$, we denote $\text{proj}_P\mathcal{W}$ the projection of $\mathcal{W}$ to $P$ and we let $\Bar{P}$ be any $n$-dimensional affine subspace such that $|\text{proj}_{\Bar{P}}\mathcal{W}| = \min\{|\text{proj}_P\mathcal{W}|: P\in Gr_{n}(\mathbb{R}^{n+m})\}$ and denote any of them by $W^* = \text{proj}_{\Bar{P}}\mathcal{W}$.  

\begin{conjecture}[Wulff inequality for minimal submanifolds]
    For any $n$-dimensional minimal submanifold $\Sigma$ in $\mathbb{R}^{n+m}$ the following Wulff inequality 
    \begin{equation}\label{eq: wulff inequality conjecture}
        \frac{P_{\Phi}({\Sigma})}{|\Sigma|^{\frac{n-1}{n}}} \geq \frac{P_{\Phi}({W^*})}{|W^*|^{\frac{n-1}{n}}}
    \end{equation}
holds, and equality holds if and only if $\Sigma$ is homothetic to some $W^{*}$ \footnote{Note that by Remark \ref{remark: projection=restriction wulff}, the right hand side of \eqref{eq: wulff inequality conjecture} is the same for any projection with minimal projection area, hence it is well-defined.}.
\end{conjecture}
In this paper, by adapting arguments of Brendle and Cabre, we first prove a boundary weighted isoperimetric inequality for minimal submanifolds in Euclidean space:
\begin{theorem}\label{aniso-peri}
    Let $\mathcal{W}$ be centrally symmetric $n+m$ dimensional convex body in $\mathbb{R}^{n+m}$, and $\Phi$ be the corresponding support function. Let $\Sigma$ be a minimal hypersurface with boundary $\partial \Sigma$ embedded in $\mathbb{R}^{n+m}$. We have
    \begin{equation}\label{opt-inequ}
        P_{\Phi}({\Sigma}) \geq n(\sup_{\mathcal{F}_{n,m}}\int_{\mathcal{W}} f)^{\frac{1}{n}} |\Sigma|^{\frac{n-1}{n}}
    \end{equation}
    where
    \begin{equation}\label{F_n,m space}
        \mathcal{F}_{n,m} = \{f \in L^1(\mathbb{R}^{n+m}): \textrm{supp}(f)\subset\mathcal{W}, f\geq 0, \int_Pf \leq 1 \,\,\text{for}\,\,P\in \textrm{Graff}_m(\mathbb{R}^{n+m})\}\footnote{We note that the space  $\mathcal{F}_{n,m}$ may not be compact. }
    \end{equation}
and $\textrm{Graff}_m(\mathbb{R}^{n+m})$ is the affine Grassmannian consisting of all $m$-dimensional affine subspaces in $\mathbb{R}^{n+m}$
\end{theorem}
\begin{remark}\label{rem opt}
 By Fubini's theorem  we have the trivial upper bound
\begin{equation}\label{trivial-bound}
  \sup_{\mathcal{F}_{n,m}}  \int_{\mathcal{W}} f  \leq |W^*|.
\end{equation}
It is also clear that the supremum is always positive, as the function $f = \chi_{B^{n+m}(\delta)}$ for  $\delta>0$ sufficiently small is in $\mathcal{F}_{n, m}$. Later in \eqref{lower bound f}, we will provide a better estimate of the quantity on left hand side of \eqref{trivial-bound} for general Wulff shapes $\mathcal{W}$, and show that in some special cases the equality in \eqref{trivial-bound} holds.
\end{remark}
To state our next result, we define a class of convex bodies generated from ellipsoids using gluing and cutting operations. We refer to this class of convex bodies as \textit{long convex bodies}, since a typical example of such shape is given by $\mathcal{W} = K \times [-T, T]^m\subset \mathbb{R}^{n+m}$ where $K$ is any $n$-dimensional centrally symmetric convex body and $2T \geq \text{diam}(K)$.

\begin{definition}[long convex body]\label{long-body-def}
    Let $\mathcal{E}_{n+m}$ be the set of all $(n+m)$-dimensional ellipsoids centering at origin. For $i = 1, 2$, let $\mathcal{W}_i$ be centrally symmetric convex bodies in $\mathbb{R}^{n+m}$, and $W_i^*$ some area-minimizing projection for $\mathcal{W}_i$. Let $P$ be an $n$-dimensional affine subspace such that $|\text{proj}_P\mathcal{W}_1| = |W_1^*|$. The set of long convex bodies $\mathcal{L}_{n,m}$ is the smallest set of centrally symmetric convex bodies satisfying the following conditions:
    \begin{itemize}
        \item $\mathcal{E}_{n+m} \subset \mathcal{L}_{n,m}$.
        \item Gluing: if $\mathcal{W}_1 \in \mathcal{L}_{n,m}$, $\mathcal{W}_2 \supset \mathcal{W}_1$ and $\text{proj}_P\mathcal{W}_1 = \text{proj}_P\mathcal{W}_2$, then $\mathcal{W}_2 \in \mathcal{L}_{n,m}$.
        \item Cutting: if $\mathcal{W}_1 \in \mathcal{L}_{n,m}$, and $\mathcal{W}_2 = (\text{proj}_P)^{-1}(\text{proj}_P\mathcal{W}_2) \cap \mathcal{W}_1$, then $\mathcal{W}_2 \in \mathcal{L}_{n,m}$.
    \end{itemize}
\end{definition}
\begin{remark}\label{remark: necessary condition for equality}
    It is clear that $\mathcal{L}_{n,m}$ consists of all centrally symmetric convex bodies that can be obtained by performing a finite number of gluing and cutting operations on an ellipsoid. Performing these operations repeatedly elongates the shape in the normal directions of the area-minimizing projection. These two operations provide necessary conditions for equality in \eqref{trivial-bound}, in the sense that performing gluing or cutting along some area-minimizing projection direction should preserve the fact that it is an area-minimizing projection direction. Specifically, if the equality in \eqref{trivial-bound} holds for some $\mathcal{W}_1$, and $\mathcal{W}_2$ (not necessarily centrally symmetric here) is generated from $\mathcal{W}_1$ by gluing or cutting with respect to some $P$ such that $\text{proj}_P\mathcal{W}_1$ is an area-minimizing projection for $\mathcal{W}_1$, then $\text{proj}_P\mathcal{W}_2$ must also be an area-minimizing projection for $\mathcal{W}_2$, and the equality in \eqref{trivial-bound} still holds for $\mathcal{W}_2$ (see Lemma \ref{lemma: necessary condition}). For example, when $\mathcal{W}$ is an $(n+1)$-dimensional cube or a short cylinder of the form $B^n \times [-\frac{\eps}{2}, \frac{\eps}{2}]$ for small $\eps>0$ in Figure \ref{projection of cylinder cubes}, it is possible to cut along some area-minimizing projection direction, such that after the cut, it is no longer the area-minimizing projection direction for the resulting shape. Hence, the equality in \eqref{trivial-bound} must be strict for the cube and short cylinder. 
\end{remark}
\begin{figure}[h!]\label{projection of cylinder cubes}
 \centering
  \includegraphics[width=1\textwidth]{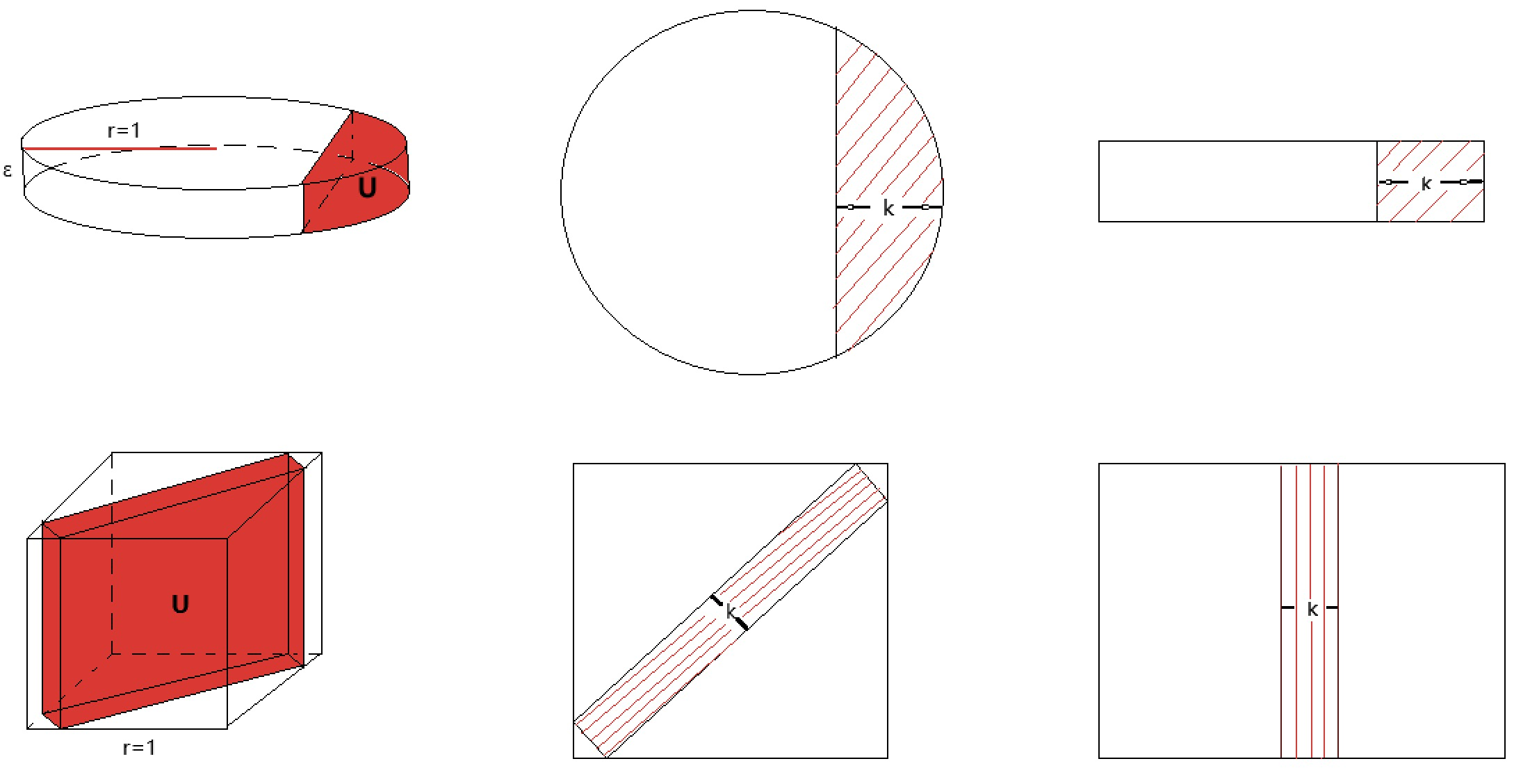}
 \caption{\textbf{Projections of short cylinder and cube.} For short cylinder, projection along any horizontal direction is area-minimizing. However, cutting along a horizontal direction sufficiently close to boundary gives a shape $U$ such that horizontal projection is no longer area-minimizing, since the vertical projection has area roughly $k^{3/2}$ and the horizontal projection has area $\eps k$. Similarly, the vertical projection was area-minimizing for cube, but such a vertical cut gives $U$, whose vertical projection has area roughly $\sqrt2 k$ and horizontal projection has area $k$.}
 \end{figure}

In the cases of codimension one and two, by designing good candidates on ellipsoids and long convex bodies, we verify the above two conjectures when $\mathcal{W} \in \mathcal{L}_{n,m}$ ($m=1, 2$) and the weights are support functions of $\mathcal{W}$. 
\begin{theorem}\label{thm: sharp cases}
  For $m=1, 2$,  let $\mathcal{W} \in \mathcal{L}_{n,m}$ and $\Phi$ be the corresponding support function, then 
    for $\Sigma$ being a minimal submanifold with boundary $\partial \Sigma$ embedded in $\mathbb{R}^{n+m}$, we have
    \begin{equation}\label{wulff-isop}
        \frac{P_{\Phi}({\Sigma})}{|\Sigma|^{\frac{n-1}{n}}} \geq \frac{P_{\Phi}({W^*})}{|W^*|^{\frac{n-1}{n}}},
    \end{equation}
    In particular, equality holds if and only if $\Sigma$ is homothetic to $W^*$.
\end{theorem}
\begin{remark}
    In higher codimensional case when $\mathcal{W}$ is an ellipsoid centering at origin and $m > 2$, we will show in Proposition \ref{codim-2} that
    \begin{equation}
        \frac{P_{\Phi}({\Sigma})}{|\Sigma|^{\frac{n-1}{n}}} \geq \left(\frac{(n+m)|B^{n+m}|}{m|B^m||B^n|}\right)^{\frac{1}{n}}\frac{P_{\Phi}({W^*})}{|W^*|^{\frac{n-1}{n}}}.
    \end{equation}
    Combine with the sharp cases in Theorem \ref{thm: sharp cases}, the isotropic (i.e. $\mathcal{W} = B^{n+m}$ and $\Phi(x) = |x|$) results in \cite{Brendle2021} can be generalized to any ellipsoid.
\end{remark}

More generally we prove a Wulff inequality for minimal submanifolds in the Euclidean space, where the constant depends only on dimension and codimension.
\begin{theorem}\label{thm: nonsharp codim 1}
    Let $\mathcal{W}$ be centrally symmetric $(n+m)$-dimensional convex body in $\mathbb{R}^{n+m}$ and $\Phi$ be the corresponding support function. Then for a minimal submanifold $\Sigma$ with boundary $\partial \Sigma$ embedded in $\mathbb{R}^{n+m}$, we have
    \begin{equation}\label{nonsharp-wulff-isop}
        \frac{P_{\Phi}({\Sigma})}{|\Sigma|^{\frac{n-1}{n}}} \geq c_{n, m}\frac{P_{\Phi}({W^*})}{|W^*|^{\frac{n-1}{n}}},
    \end{equation}
    where 
    \begin{equation}
        c_{n,m}=\begin{cases}
            \frac{1}{\sqrt{n+m}} & m = 1,2 ,\\
            \frac{1}{\sqrt{n+m}}\left(\frac{(n+m)|B^{n+m}|}{m|B^m||B^n|}\right)^{\frac{1}{n}} & m \geq 3
        \end{cases}
    \end{equation}
    is a constant only depending on $m$ and $n$ but independent of the weight $\Phi$.\footnote{The Wulff inequality is scaling-invariant for weights or $(n+m)$-dimensional Wulff shape, but the space of $(n+m)$-dimensional Wulff shapes, after taking quotient by the scaling symmetry, is non-compact.}
\end{theorem}
\begin{remark}\label{Fc}
It would be very interesting to improve the constant $c_{n, m}$ in \eqref{nonsharp-wulff-isop} to some constant independent of $m$ and $n$, which is expected to be 1. Another interesting direction is to prove sharp minimal submanifolds Wulff inequality for more general class of weights or the associated convex bodies.
\end{remark}
The paper is organized as follows: in Section 2, we prove Theorem \ref{aniso-peri}; in Section 3, we prove Theorem \ref{thm: sharp cases}; and in Section 4, we prove Theorem \ref{thm: nonsharp codim 1}.

\textbf{Acknowledgments.}
The authors appreciate the funding and research environment support from MIT, UCSD, and ITS Westlake University, respectively.

\section{Basic Convex Geometry}
In this section, we first recall some basic properties of convex geometry that will be used in later sections.

Let $P$ be an $n$-dimensional subspace passing through the origin. We now consider the following two sets in $P$:
\begin{enumerate}
    \item $W_P$ is the Wulff shape induced by $\Phi|_P$,
    \item $\text{proj}_P\mathcal{W}$ is the projection of $\mathcal{W}$ onto $P$;
\end{enumerate}
The following lemma shows that the two objects are the same.
\begin{lemma}\label{proj-wulff}
    If $\Phi$ is the support function of $\mathcal{W}$, then $W_P = \text{proj}_P\mathcal{W}$, and $\Phi|_P$ is the support function of $W_P$ in $P$.
\end{lemma}
\begin{proof}


    It suffices to show that $\Phi|_P$ is the support function of $\text{proj}_P\mathcal{W}$, since support functions and Wulff shapes uniquely determine each other (see \cite[page 2976]{cabre2016sharp}). For any $\nu \in P\cap S^{n+m-1}$ and $x \in \mathcal{W}$, we have $x \cdot \nu = \text{proj}_Px \cdot \nu$, hence
    \begin{equation}
        \Phi|_P(\nu) = \sup\{x \cdot \nu: x\in \mathcal{W}\} = \sup\{\text{proj}_Px \cdot \nu: x \in \mathcal{W}\} = \sup\{y \cdot \nu: y \in \text{proj}_P\mathcal{W}\}.
    \end{equation}
\end{proof}

\begin{remark}\label{remark: projection=restriction wulff}
    Let $P$ be any $n$-dimensional affine subspace such that $|\text{proj}_P\mathcal{W}|$ is minimized, and let $P'$ be the $n$-dimensional subspace parallel to $P$ and passing through the origin. As a consequence of the lemma, $\Phi|_{P'}$ is also the support function of $\text{proj}_P\mathcal{W}$ in $P$, so
    \begin{align}\label{Wulff-const}
        P_{\Phi}(\text{proj}_P\mathcal{W}) &= \int_{\partial \text{proj}_P\mathcal{W}}\Phi|_{P'}(\nu(x))dS \notag \\
        &= \int_{\partial \text{proj}_P\mathcal{W}}x \cdot \nu(x)dS \notag \\
        &= \int_{\text{proj}_P\mathcal{W}}\text{div}(x)dx \notag \\
        &=n|\text{proj}_P\mathcal{W}|,
    \end{align}
    where we used the fact that for almost every $x \in \partial \text{proj}_P\mathcal{W}$, the unit outward normal $\nu(x)$ exists, and $x \cdot \nu(x) = \Phi|_{P'}(x)$ because it $\Phi|_{P'}$ is the support function. According to \eqref{Wulff-const}, for any $P_1$ and $P_2$ such that both $|\text{proj}_{P_1}\mathcal{W}|$ and $|\text{proj}_{P_2}\mathcal{W}|$ are minimized, it follows that
    \begin{align}
        \frac{P_{\Phi}(\text{proj}_{P_1}\mathcal{W})}{|\text{proj}_{P_1}\mathcal{W}|^{\frac{n-1}{n}}} = \frac{P_{\Phi}(\text{proj}_{P_2}\mathcal{W})}{|\text{proj}_{P_2}\mathcal{W}|^{\frac{n-1}{n}}}.
    \end{align}
    Thus, the right-hand sides of \eqref{eq: wulff inequality conjecture}, \eqref{wulff-isop} and \eqref{nonsharp-wulff-isop} are well defined for given $\mathcal{W}$.
\end{remark}

\section{Proof of Theorem \ref{aniso-peri}}
In this section, we prove Theorem \ref{aniso-peri}. Let $\Sigma\subset \mathbb{R}^{n+m}$ be a compact $n$-dimensional minimal submanifold with boundary $\partial \Sigma$. We first consider the special case that $\Sigma$ is connected with boundary $\partial \Sigma$. By solving $u : 
\Sigma\to \mathbb{R}$ from the following Poisson equation with Neumann boundary condition
\begin{align}\label{PDE neumann}
    \begin{cases}
    \Delta_{\Sigma} u=\frac{P_\Phi(\Sigma)}{|\Sigma|} & \text { in } \Sigma, \\ \langle\nabla^{\Sigma}u, \nu\rangle=\Phi(\nu) & \text { on } \partial \Sigma,\end{cases}
\end{align}
we can construct the map:
\begin{equation}
    T:N\Sigma \xrightarrow{} \mathbb{R}^{n+1}, (x, y) \mapsto \nabla^{\Sigma} u(x)+y
\end{equation}
where $N\Sigma$ is the normal bundle of $\Sigma$. As in \cite{Brendle2021}, we define the following set:
\begin{equation}
 A=\{(x, y)\in N(\Sigma\backslash \partial\Sigma): \det (D^{2}_{\Sigma} u-\langle \textrm{I\!I}_{\Sigma}(x), y \rangle)\geq 0\},
\end{equation}
where $\textrm{I\!I}_{\Sigma}$ is the second fundamental form of $\Sigma\subset \mathbb{R}^{n+m}$ 
\begin{lemma}\label{surjective}
    We have $T(A) \supset \mathcal{W}^{int}$, thus $T$ is a surjective map from $A$ onto the interior of the Wulff shape $\mathcal{W}$.
\end{lemma}
\begin{proof}
    Let $\xi \in \mathcal{W}^{int}$ and we consider the map $w: \Sigma \xrightarrow{} \mathbb{R}$ defined by
    \begin{align}
        w(x) = u(x)-\langle x, \xi \rangle.
    \end{align}
    Denote $\nu(x)$ the unit outward normal vector on the boundary $ \partial \Sigma$ relative to $\Sigma$. Then, for any $x \in \partial \Sigma$,
    \begin{equation}
        \langle \nabla^{\Sigma}w(x), \nu(x)\rangle = \langle \nabla^{\Sigma}u(x), \nu(x)\rangle - \langle \xi, \nu(x)\rangle = \Phi(\nu) - \langle \xi, \nu(x)\rangle > 0.
    \end{equation}
    Thus $w$ attains its minimum at some interior point $\Bar{x}$. It follows that $\nabla^{\Sigma}w(\Bar{x}) = 0$, which implies 
    \begin{align}\label{xi-equ}
        \xi = \nabla^{\Sigma}u(\Bar{x})+y
    \end{align}
    for some $y \in N_{\Bar{x}}\Sigma$. It remains  to compute the determinant. Since $w$ achieves its minimum, we have
    \begin{equation}
        0 \leq D^2_{\Sigma}w(\Bar{x}) = D^2_{\Sigma}u(x) - \langle \textrm{I\!I}_{\Sigma}(\Bar{x}), \xi\rangle = D^2_{\Sigma}u(x) - \langle \textrm{I\!I}_{\Sigma}(\Bar{x}), y\rangle,
    \end{equation}
    where the last equality holds from \eqref{xi-equ}.
    Therefore, we obtain $T(\Bar{x}, y) = \xi$ and $(\Bar{x}, y) \in A$.
\end{proof}

\begin{lemma}\label{det-T}
    For all $(x, y) \in A$, the determinant of transport map $T$ satisfies
    \begin{equation}
        0 \leq \det DT(x, y) \leq \left(\frac{P_\Phi(\Sigma)}{n|\Sigma|}\right)^n.
    \end{equation}
\end{lemma}
\begin{proof}
    As shown in \cite[Lemma 5]{Brendle2021}, one can show that 
    \begin{align}
        \det DT(x, y) = \det(D^{2}_{\Sigma} u-\langle \textrm{I\!I}_{\Sigma}(x), y \rangle).
    \end{align}
    By the definition of $A$, we have $\det(D^{2}_{\Sigma} u-\langle \textrm{I\!I}_{\Sigma}(x), y \rangle) \geq 0$. Thus, by applying the arithmetic-geometric mean inequality, the vanishing mean curvature property of minimal submanifolds, and \eqref{PDE neumann}, we obtain
    \begin{equation}
        \det(D^{2}_{\Sigma} u-\langle \textrm{I\!I}_{\Sigma}(x), y \rangle) \leq \left( \frac{\text{tr}(D^{2}_{\Sigma} u-\langle \textrm{I\!I}_{\Sigma}(x), y \rangle)}{n}\right)^n = \left(\frac{P_\Phi(\Sigma)}{n|\Sigma|}\right)^n.
    \end{equation}
\end{proof}

\begin{proof}[Proof of Theorem \ref{aniso-peri}]

We recall
\begin{equation}
   f\in  \mathcal{F}_{n,m} = \{f \in L^1(\mathbb{R}^{n+m}): \textrm{supp}(f)\subset\mathcal{W}, f\geq 0, \int_Pf \leq 1 \,\,\text{for}\,\,P\in \textrm{Graff}_m(\mathbb{R}^{n+m})\},
\end{equation}
and by Lemma \ref{surjective} and Lemma \ref{det-T}, we have
\begin{align}\label{eq: inequality1}
    \int_{\mathcal{W}} f(\xi) d\xi &\leq \int_{\Sigma}\int_{N_x\Sigma}f(T(x, y))|\det DT(x, y)|1_{A}(x, y)dy dx\\
    &\leq \left(\frac{P_\Phi(\Sigma)}{n|\Sigma|}\right)^n \int_{\Sigma}\int_{N_x\Sigma}f(\nabla^{\Sigma}u(x)+y) dy dx \notag\\
    &\leq \left(\frac{P_\Phi(\Sigma)}{n|\Sigma|}\right)^n |\Sigma| \notag\\
    &= \left(\frac{P_\Phi(\Sigma)}{n|\Sigma|^{\frac{n-1}{n}}}\right)^n.\notag
\end{align}
In particular, the above inequality gives
\begin{equation}
    \sup_{\mathcal{F}_{n,m}}\int_{\mathcal{W}} f \leq \left(\frac{P_\Phi(\Sigma)}{n|\Sigma|^{\frac{n-1}{n}}}\right)^n,
\end{equation}
which implies \eqref{opt-inequ}. It remains to consider the case where $\Sigma$ is disconnected. In that case, we apply the inequality to each individual connected component of $\Sigma$, sum over all connected components, and use the strict inequality
\begin{align}\label{component-inequ}
    a^{\frac{n-1}{n}}+b^{\frac{n-1}{n}}>a(a+b)^{-\frac{1}{n}}+b(a+b)^{-\frac{1}{n}}=(a+b)^{\frac{n-1}{n}}
\end{align}
for all $a, b>0$.

\end{proof}

\section{Proof of Theorem \ref{thm: sharp cases}}
In  this section, we prove
Theorem \ref{thm: sharp cases}. By Theorem \ref{aniso-peri} and \eqref{Wulff-const}, it suffices to show that there exists some $\tilde{f} \in \mathcal{F}_{n,m}$ such that 
\begin{align}\label{tilde-f}
    \int_{\mathcal{W}} \tilde{f} = |W^*|,
\end{align}
or there exists a sequence $f_\sigma \in \mathcal{F}_{n,m}$ such that 
\begin{align}\label{f-sigma}
    \lim_{\sigma\to 1}\int_{\mathcal{W}} f_\sigma = |W^*|.
\end{align}
In Proposition \ref{codim-1}, Proposition \ref{codim-2} and Proposition \ref{long-body}, we demonstrate that when $m = 1, 2$, equation \eqref{tilde-f} or \eqref{f-sigma} holds when $\mathcal{W}$ is a centrally symmetric $(n+m)$-dimensional ellipsoid and long convex body. In Proposition \ref{codim-2} we also obtain a weaker estimate for $m > 2$. Then in Proposition \ref{equality}, we discuss the sharpness and rigidity when equality holds for ellipsoids and long bodies and $m=1, 2$.

\begin{proposition}\label{codim-1}
   If codimension $m=1$ and $\mathcal{W} \in \mathcal{E}_{n+1}$,  there exists some $\tilde{f} \in \mathcal{F}_{n,1}$ such that \eqref{tilde-f} holds. Hence
   \begin{equation}
       \sup_{f \in \mathcal{F}_{n,1}}\int_{\mathcal{W}}f = |W^*|.
   \end{equation}
\end{proposition}

\begin{proof}
Suppose $\mathcal{W}$ is an ellipsoid given by
\begin{equation}
    \frac{x_1^2}{\lambda_1^2}+...+\frac{x_{n+1}^2}{\lambda_{n+1}^2} \leq 1
\end{equation}
where $0 < \lambda_1 \leq ... \leq \lambda_{n+1}$. Denote
\begin{align}
    \Lambda = 
    \begin{pmatrix}
        \lambda_1^{-1} & & & & \\
         & \lambda_2^{-1} & & & \\
         & & \ddots & &\\
         & & & & \lambda_{n+1}^{-1} \\
    \end{pmatrix}.
\end{align}
Let the function $f(x)$ be defined by
\begin{equation}
    f(x) = \frac{1}{\sqrt{1- |\Lambda x|^2}}1_{\mathcal{W}}(x).
\end{equation}
Consider a line $l(t) = t\boldsymbol{\alpha} + \boldsymbol{\omega}$ that passes through the ellipsoid along unit vector $\boldsymbol{\alpha}$ direction. Then the end points $T_i\boldsymbol{\alpha}+\boldsymbol{\omega}$ ($i=1, 2$) of the chord in the ellipsoid satisfies
\begin{equation}
    |\Lambda(T_i\boldsymbol{\alpha}+\boldsymbol{\omega})|^2 = 1, \quad i=1, 2.
\end{equation}
This is a quadratic equation
\begin{equation}
    -aT_i^2+bT_i+c = 0, \quad i=1, 2,
\end{equation}
where
\begin{equation}
    a = |\Lambda\boldsymbol{\alpha}|^2, \quad b = -2\langle \Lambda\boldsymbol{\alpha}, \boldsymbol{\omega}\rangle, \quad c = 1-|\Lambda\boldsymbol{\omega}|^2.
\end{equation}
Upon solving this, we obtain
\begin{equation}
    (\frac{2aT_i-b}{\sqrt{b^2+4ac}})^2 = 1.
\end{equation}
Thus, the integral over the chord is given by
\begin{align}
    &\int_{T_1}^{T_2}\frac{1}{\sqrt{1 - |\Lambda(t\boldsymbol{\alpha} + \boldsymbol{\omega})|^2}}dt \notag\\
    &=\int_{T_1}^{T_2}\frac{1}{\sqrt{-at^2+bt+c}}dt\notag\\
    &=\left. \frac{1}{\sqrt{a}}\arcsin(\frac{2at-b}{\sqrt{b^2+4ac}})\right|_{T_1}^{T_2} \notag\\
    &= \frac{\pi}{\sqrt{a}}.
\end{align}
We notice that
\begin{equation}
    \frac{\pi}{\sqrt{a}} = \frac{\pi}{|\Lambda\boldsymbol{\alpha}|} \leq \lambda_{n+1}\pi,
\end{equation}
and the equality holds if and only if $\boldsymbol{\alpha} = (0,...,0, 1)$, that is the line is parallel to the longest axis. We now define the function $\tilde f$ as
\begin{equation}\label{eq: codim1 ellipsoid}
    \tilde{f} = \frac{f}{\pi\lambda_{n+1}}\geq 0.
\end{equation}
It satisfies its integral over any line that passes through the ellipsoid is at most 1, and achieves 1 when the line is parallel to the longest direction of ellipsoid, that is the axis corresponding to $\lambda_{n+1}$. Moreover, Fubini's Theorem gives
\begin{equation}
    \int_{\mathcal{W}} \tilde{f} = \int_{\text{proj}_{\{x_{n+1}=0\}}\mathcal{W}}\int_{\gamma_{x_{n+1}}} \tilde{f}(x)dx'dx_{n+1} = |\text{proj}_{\{x_{n+1}=0\}}\mathcal{W}| = |W^*|.
\end{equation}
Thus, $\tilde{f}\in \mathcal{F}_{n,1}$ and satisfies \eqref{tilde-f}, and we have completed the proof.
\end{proof}

For codimension $m\geq 2$ cases, we have following proposition.

\begin{proposition}\label{codim-2}
   If codimension $m = 2$ and $\mathcal{W} \in \mathcal{E}_{n+m}$, then there exists a sequence of functions $f_\sigma \in \mathcal{F}_{n,m}$ such that \eqref{f-sigma} holds. Hence when $m = 2$,
   \begin{equation}
       \sup_{f \in \mathcal{F}_{n,2}}\int_{\mathcal{W}}f = |W^*|.
   \end{equation}
   If codimension $m > 2$ and $\mathcal{W} \in \mathcal{E}_{n+m}$, then there exists a sequence of functions $f_\sigma \in \mathcal{F}_{n,m}$ such that
   \begin{equation}
       \sup_{f \in \mathcal{F}_{n,m}}\int_{\mathcal{W}}f \geq \lim_{\sigma\to 1}\int_{\mathcal{W}}f_{\sigma} = \frac{(n+m)|B^{n+m}|}{m|B^m||B^n|}|W^*|.
   \end{equation}
\end{proposition}

\begin{proof}
Denote $\mathcal{W} = E_{\lambda_1, \lambda_2,\cdots,\lambda_{n+m}}$ the $(n+m)$-dimensional ellipsoid
\begin{align}
    \frac{x_1^2}{\lambda_1^2}+\frac{x_2^2}{\lambda_2^2}+\cdots+\frac{x_{n+m}^2}{\lambda_{n+m}^2} \leq 1,
\end{align}
where $\lambda_1\leq \lambda_2 \leq \cdots \leq \lambda_{n+m}$, and $U_{\sigma} = E_{\lambda_1,\lambda_2,\cdots,\lambda_{n+m}}\setminus E_{\sigma\lambda_1,\sigma\lambda_2,\cdots,\sigma\lambda_{n+m}}$ for some $0< \sigma < 1$.

Let  $P = \boldsymbol{q}+t_1 \boldsymbol{r_1}+t_2 \boldsymbol{r_2} + \cdots + t_{m} \boldsymbol{r_m}$ be an $m$-dimensional affine subspace which intersects with $U_{\sigma}$, where $\{\boldsymbol{r_j}\}_{j=1}^m$ are orthonormal. The intersection region $U_{\sigma} \cap P$ is the shell of two $m$-dimensional homothetic ellipsoids (if $P$ is very close to the boundary, it would be fully outside of $E_{\sigma\lambda_1,\cdots,\sigma\lambda_{n+m}}$, so in that case $U_{\sigma} \cap P = E_{\lambda_1,\cdots,\lambda_{n+m}} \cap P$ is just an ellipsoid). Denote
\begin{align}
    \Lambda = 
    \begin{pmatrix}
        \lambda_1^{-1} & & & & \\
         & \lambda_2^{-1} & & & \\
         & & \ddots & &\\
         & & & & \lambda_{n+m}^{-1} \\
    \end{pmatrix}.
\end{align}
Note that $\Lambda(E_{\lambda_1,\cdots,\lambda_{n+m}}) = B^{n+m}$, so we can study the intersection between ellipsoid and lower dimensional plane by converting it to an intersection between ball and plane via $\Lambda$, and converting everything back in the end by applying appropriate scaling. (For a detailed computation for $m=1$ and $n=2$, see \cite{klein2012ellipsoid}.)

Specifically, the above computation yields that the length of the $i$-th axis of the outer ellipsoid $P \cap E_{\lambda_1,\cdots,\lambda_{n+m}}$ is
\begin{align}
  \frac{\sqrt{1-d}}{|\Lambda \boldsymbol{r_i}|},\quad 1\leq i\leq m,
\end{align}
where $d = |\Lambda \boldsymbol{q}|^2 - \frac{\langle\Lambda \boldsymbol{q}, \Lambda \boldsymbol{r_1}\rangle^2}{|\Lambda \boldsymbol{r_1}|^2} - \frac{\langle\Lambda \boldsymbol{q}, \Lambda \boldsymbol{r_2}\rangle^2}{|\Lambda \boldsymbol{r_2}|^2} - \cdots - \frac{\langle\Lambda \boldsymbol{q}, \Lambda \boldsymbol{r_m}\rangle^2}{|\Lambda \boldsymbol{r_m}|^2}$ is the squared distance between origin and $\Lambda P$. The length of $i$-th axis of inner ellipsoid $P \cap E_{\sigma \lambda_1,\cdots,\sigma\lambda_{n+m}}$ is
\begin{align}
  \frac{\sqrt{\sigma^2-d}}{|\Lambda \boldsymbol{r_i}|},\quad 1\leq i\leq m.
\end{align}
As discussed before, $U_{\sigma} \cap P$ is the shell formed by two homothetic ellipsoids only when $P$ cuts through $E_{\sigma\lambda_1,\cdots,\sigma\lambda_{n+m}}$, which is the same as $\Lambda P$ cuts through $B^n(\sigma)$, or equivalently $d \leq \sigma^2$. Hence the intersection volume can be expressed as
\begin{equation}
    |U_{\sigma} \cap P| = |B^m|[(1-d)^{m/2} - (\sigma^2-d)^{m/2}_+]\prod_{j=1}^m |\Lambda \boldsymbol{r_j}|^{-1}.
\end{equation}
Since $\{\boldsymbol{r_j}\}_{j=1}^m$ is an orthonormal basis, we have
\begin{equation}
    \prod_{j=1}^m|\Lambda \boldsymbol{r_j}|^{-1} \leq \prod_{j=1}^m \lambda_{n+m+1-j}.
\end{equation}
Then it follows
\begin{equation}
    |U_{\sigma} \cap P| \leq |B^m|\frac{m}{2}(1-\sigma^2)\prod_{j=1}^m \lambda_{n+m+1-j} =: C_{\sigma}.
\end{equation}
Now we consider  
\begin{equation}\label{eq: codim2 ellipsoid}
    f_{\sigma} = \frac{1}{C_{\sigma}}1_{U_{\sigma}},
\end{equation} 
and it clearly satisfies $\text{supp}(f_{\sigma}) \subset \mathcal{W}$ and $f_{\sigma} \geq 0$. For any $m$-dimensional affine subspace $P$ in $\mathbb{R}^{n+m}$,
\begin{equation}
    \int_P f_{\sigma} = \frac{|U_{\sigma} \cap P|}{C_{\sigma}} \leq 1.
\end{equation}

On the other hand, the volume of the shell $U_{\sigma}$ is given by the difference of two $(n+m)$-dimensional ellipsoids, so
\begin{align}
    \int_{\mathcal{W}} f_{\sigma} &= \frac{|U_\sigma|}{C_{\sigma}} = \frac{1}{C_{\sigma}}|B^{n+m}|(1-\sigma^{n+m})\prod_{j=1}^{n+m}\lambda_j\notag\\
    &= \frac{2(1-\sigma^{n+m})}{1-\sigma^2}\frac{|B^{n+m}|}{m|B^m|} \prod_{j=1}^n\lambda_j.
\end{align}
Taking the limit as $\sigma \xrightarrow{} 1$, we obtain that
\begin{equation}
    \lim_{\sigma\to 1}\int_{\mathcal{W}} f_{\sigma} = \frac{(n+m)|B^{n+m}|}{m|B^m|} \prod_{j=1}^n\lambda_j = \frac{(n+m)|B^{n+m}|}{m|B^m||B^n|} |W^*|,
\end{equation}
where we used $|W^*| = |B^n|\prod_{j=1}^n\lambda_j$ for the ellipsoid. In particular, when $m = 2$, the coefficient simplifies to 1, and we have completed the proof of \eqref{f-sigma}.
\end{proof}

Next we will prove a lemma about the necessary condition discussed in Remark \ref{remark: necessary condition for equality}. We will use this lemma and the $\tilde{f}$ and $f_{\sigma}$ constructed above to verify \eqref{wulff-isop} for general $\mathcal{W} \in \mathcal{L}_{n,m}$ and $m = 1, 2$. 


\begin{lemma}\label{lemma: necessary condition}
    Suppose $\mathcal{W}_1$ and $f_{\sigma} \in\mathcal{F}^{\mathcal{W}_1}_{n,m}$ satisfies
    \begin{equation}
        \sup_{f \in \mathcal{F}^{\mathcal{W}_1}_{n,m}}\int_{\mathcal{W}_1}f = \lim_{\sigma\to 1}\int_{\mathcal{W}_1}f_{\sigma} = |W_1^*|,
    \end{equation}
    and let $\mathcal{W}_2$ be generated from $\mathcal{W}_1$ via gluing or cutting with respect to an $n$-dimensional affine subspace $P$ such that $\text{proj}_P\mathcal{W}_1$ is an area-minimizing projection for $\mathcal{W}_1$. Then, $\text{proj}_P\mathcal{W}_2$ must also be an area-minimizing projection for $\mathcal{W}_2$, and 
    \begin{equation}
        \sup_{f \in \mathcal{F}^{\mathcal{W}_2}_{n,m}}\int_{\mathcal{W}_2}f = \lim_{\sigma\to 1}\int_{\mathcal{W}_2}f_{\sigma}|_{\mathcal{W}_2} = |W_2^*|.
    \end{equation}
\end{lemma}
\begin{proof}
    Suppose $\mathcal{W}_2 \supset \mathcal{W}_1$ is obtained via a gluing operation. Then by Definition \ref{long-body-def}, $\mathcal{F}^{\mathcal{W}_1}_{n,m} \subset \mathcal{F}^{\mathcal{W}_2}_{n,m}$. Moreover, we also have 
    \begin{align}
        |W_2^*| \geq |W_1^*| = |\text{proj}_P\mathcal{W}_1|= |\text{proj}_P\mathcal{W}_2| \geq |W_2^*|.
    \end{align}
    Thus, we conclude that $|W_1^*| = |W_2^*|$, and $\text{proj}_P\mathcal{W}_2$ is an area-minimizing projection for $\mathcal{W}_2$. By assumption, we can thus find a sequence of $f_{\sigma} \in \mathcal{F}^{\mathcal{W}_1}_{n,m} \subset \mathcal{F}^{\mathcal{W}_2}_{n,m}$ such that
    \begin{equation}
        \lim_{\sigma \to 1}\int_{\mathcal{W}_2}f_{\sigma} = \lim_{\sigma \to 1}\int_{\mathcal{W}_1}f_{\sigma} = |W_1^*| = |W_2^*|,
    \end{equation}
    which finishes the proof for gluing operation.
    
    Now suppose $\mathcal{W}_2 \subset \mathcal{W}_1$ is obtained via a cutting operation. By Definition \ref{long-body-def}, $\text{proj}_P\mathcal{W}_1 \supset \text{proj}_P\mathcal{W}_2$. For any $x \in \text{proj}_P\mathcal{W}_1$, denote $Q_x = (\text{proj}_P)^{-1}x$ the $m$-dimensional affine subspace orthogonal to $\text{proj}_P\mathcal{W}_1$ and passing through $x$. By assumption, we can find a sequence $f_{\sigma} \in \mathcal{F}^{\mathcal{W}_1}_{n,m}$ such that
    \begin{equation}
        \lim_{\sigma \to 1}\int_{\mathcal{W}_1}f_{\sigma} = |W_1^*|.
    \end{equation}
    Use the fact that $\text{proj}_P\mathcal{W}_1$ is an area-minimizing projection, we must have for a.e. $x \in \text{proj}_P\mathcal{W}_1$,
    \begin{equation}
        \lim_{{\sigma \to 1}}\int_{Q_x \cap \mathcal{W}_1} f_{\sigma} = 1.
    \end{equation}
    On the other hand, for any $x \in \text{proj}_P\mathcal{W}_2$, by definition of cutting
    \begin{equation}
        \mathcal{W}_2 = (\text{proj}_P)^{-1}(\text{proj}_P\mathcal{W}_2) \cap \mathcal{W}_1 \supset (\text{proj}_P)^{-1}x \cap \mathcal{W}_1 = Q_x \cap \mathcal{W}_1
    \end{equation}
    Certainly $Q_x \cap \mathcal{W}_2 \subset Q_x \cap \mathcal{W}_1$, so $Q_x \cap \mathcal{W}_2 = Q_x \cap \mathcal{W}_1$ for any $x \in \text{proj}_P\mathcal{W}_2$.  Therefore, $f_{\sigma}|_{\mathcal{W}_2} \subset \mathcal{F}^{\mathcal{W}_2}_{n,m}$ satisfies
    \begin{equation}
         \lim_{{\sigma \to 1}}\int_{\mathcal{W}_2}f_{\sigma} =  \lim_{{\sigma \to 1}}\int_{\text{proj}_P\mathcal{W}_2}\int_{Q_x \cap \mathcal{W}_2} f_{\sigma} =  \lim_{{\sigma \to 1}}\int_{\text{proj}_P\mathcal{W}_2}\int_{Q_x \cap \mathcal{W}_1} f_{\sigma} =|\text{proj}_P\mathcal{W}_2|
    \end{equation}
     Combining with \eqref{trivial-bound}, we obtain
    \begin{equation}
        |W_2^*| \leq |\text{proj}_P\mathcal{W}_2| = \lim_{\sigma\to 1}\int_{\mathcal{W}_2}f_{\sigma}|_{\mathcal{W}_2} \leq \sup_{f \in \mathcal{F}^{\mathcal{W}_2}_{n,m}}\int_{\mathcal{W}_2}f_{\sigma} \leq |W_2^*|.
    \end{equation}
    Hence $\text{proj}_P\mathcal{W}_2$ must be area-minimizing projection, which finishes the proof for cutting operation.
\end{proof}

\begin{proposition}\label{long-body}
    Suppose $\mathcal{W} \in \mathcal{L}_{n,m}$ is a long convex body, then there exists a function $\tilde{f} \in \mathcal{F}_{n,m}$ when $m=1$, and a sequence of functions $f_\sigma \in \mathcal{F}_{n,m}$ when $m=2$, such that \eqref{tilde-f} and \eqref{f-sigma} hold, respectively. Hence
    \begin{equation}
        \sup_{f \in \mathcal{F}_{n,m}}\int_{\mathcal{W}}f = |W^*|.
    \end{equation}
\end{proposition}

\begin{proof}
By the discussion after Definition \ref{long-body-def}, any $\mathcal{W} \in \mathcal{L}_{n,m}$ can be constructed through a finite number of gluing and cutting operations starting from an ellipsoid. It is also clear that Proposition \ref{long-body} holds for ellipsoids by Proposition \ref{codim-1} and \ref{codim-2}. Then for any long convex body, Proposition \ref{long-body} directly follows from Lemma \ref{lemma: necessary condition}. In particular, let $E$ be the ellipsoid that generated $\mathcal{W}$, then for $m=1$ we can choose $\tilde{f}$ to be the one in Proposition \ref{codim-1}, and for $m=2$ we can choose $f_{\sigma}$ to be the sequence in Proposition \ref{codim-2}, both restricted to $\mathcal{W}$.
\end{proof}


Next, we establish the sharpness and rigidity properties when equality holds for $\mathcal{W} \in \mathcal{L}_{n,m}$ and codimension $m=1, 2$.

\begin{proposition}\label{equality}
    Let codimension $m = 1, 2$ and $\mathcal{W} \in \mathcal{L}_{n,m}$. Then the equality in \eqref{wulff-isop} holds if and only if $\Sigma$ is homothetic to $W^*$.
\end{proposition}

\begin{proof}
 It is clear from Remark \ref{remark: projection=restriction wulff} that if $\Sigma$ is homothetic to some $W^*$ then equality holds, so we only need to prove the converse.
 
First of all, if equality \eqref{wulff-isop} holds for some $\Sigma$, by \eqref{component-inequ} $\Sigma$ must be connected. Moreover, from Proposition \ref{long-body} we know $\sup_{\mathcal{F}_{n,m}}  \int_{\mathcal{W}} f=|W^*|$, and $\Sigma$ achieves the equality of \eqref{wulff-isop} if and only if $\Sigma$ achieves the equality of \eqref{opt-inequ}.

When $m=1$, let $\tilde{f}\in \mathcal{F}_{n, m}$ be the density function found in Proposition \ref{long-body} such that 
\begin{equation}
  \int_{\mathcal{W}}  \tilde{f}= \sup_{\mathcal{F}_{n,m}}  \int_{\mathcal{W}} f=|W^*|.
\end{equation}
For such $\tilde{f}$, the equalities in \eqref{eq: inequality1} must hold, meaning that
\begin{align}
    |\det DT(x, y)| = \left(\frac{P_{\Phi}(\Sigma)}{n|\Sigma|}\right)^n ~ \text{for}~ (x, y) \in A ~\text{a.e.}.
\end{align}
Moreover, we have $|\pi (A)| = |\Sigma|$ where $\pi(A) = \{x: |\{y:(x, y) \in A\}| > 0\}$ is the projection onto $\Sigma$.
Thus, the arithmetic-geometric mean inequality in Lemma \ref{det-T} must be equality for a.e. points in $A$, which only happens when $D_{\Sigma}^2u(x) - \langle \textrm{I\!I}_{\Sigma}(x), y \rangle = \frac{P_{\Phi}(\Sigma)}{n|\Sigma|}\textrm{Id}_n$ is a constant multiple of identity matrix. Thus $\textrm{I\!I}_{\Sigma}(x) = 0$ a.e. in $\pi(A)$, equivalent to a.e. in $\Sigma$. Since $\Sigma$ is smooth, it is in some codimension $m$ affine subspace $P$. By codimension zero Wulff inequality in Theorem \ref{wulff-inequ}, $\Sigma$ is homothetic to the Wulff shape corresponding to the restricted weight $\Phi|_P$. Then by Lemma \ref{proj-wulff}, $\Sigma$ must be homothetic to $\text{proj}_P\mathcal{W}$. By assumption $\Sigma$ achieves equality, so by Remark \ref{remark: projection=restriction wulff} we have
\begin{equation}
    \frac{P_{\Phi}(\Sigma)}{|\Sigma|^{\frac{n-1}{n}}} = \frac{P_{\Phi}(\text{proj}_P\mathcal{W})}{|\text{proj}_P\mathcal{W}|^{\frac{n-1}{n}}} = n|\text{proj}_P\mathcal{W}|^{\frac{1}{n}} \geq n|W^*|^{\frac{1}{n}}.
\end{equation}
On the other hand, by \eqref{trivial-bound} and the equality of \eqref{opt-inequ}, we also have 
\begin{align}
    \frac{P_{\Phi}(\Sigma)}{|\Sigma|^{\frac{n-1}{n}}} \leq n|W^*|^{\frac{1}{n}}.
\end{align}
Therefore, $|\text{proj}_P\mathcal{W}| = |W^*|$ and $\Sigma$ is homothetic to some area-minimizing projection $W^*$ in the codimension one case.

When $m=2$, let the sequence $f_\sigma\in \mathcal{F}_{n,m}$ be the corresponding approximating density functions defined in Proposition \ref{long-body}, these functions satisfy
\begin{align}\label{lim-sup}
    \lim_{\sigma\to 1}\int_{\mathcal{W}} f_\sigma = \sup_{\mathcal{F}_{n,m}}\int_{\mathcal{W}} f = |W^*|.
\end{align}
Denote $S = \cap_{\sigma} \text{supp}(f_{\sigma})$. If $\mathcal{W}$ is an ellipsoid we simply have $S = \partial \mathcal{W}$, and if $\mathcal{W}$ is a general long convex body, $S=\mathcal{W} \cap \partial E^{n+2}$, where $E^{n+2}$ is the original ellipsoid that generates $\mathcal{W}$.  As shown in \cite[Section 3]{Brendle2021}, we claim 
\begin{claim}\label{ellip-partial}
    For all $x\in \Sigma$ and all $y\in N_x\Sigma$ satisfying $T(x,y)\in S$, the equality $D_\Sigma^2 u(x)-\langle \textrm{I\!I}_{\Sigma}(x), y \rangle = \frac{P_{\Phi}(\Sigma)}{n|\Sigma|}\textrm{Id}_n$ holds.
\end{claim}

\begin{proof}
Suppose that claim fails at some $x_0\in \Sigma$ and $y_0\in N_{x_0}\Sigma$ satisfying   $T(x_0, y_0)\in S$ and $(x_0, y_0) \in A$. Then we have $D_\Sigma^2 u(x_0)-\langle \textrm{I\!I}_{\Sigma}(x_0), y_0 \rangle \neq \frac{P_{\Phi}(\Sigma)}{n|\Sigma|}\textrm{Id}_n$. The arithmetic-geometric inequality and \eqref{PDE neumann} imply
    \begin{equation}
        \det(D^{2}_{\Sigma} u(x_0)-\langle \textrm{I\!I}_{\Sigma}(x_0), y_0 \rangle) <\left(\frac{P_\Phi(\Sigma)}{n|\Sigma|}\right)^n.
    \end{equation}
    By continuity, there exists $\varepsilon\in(0,1)$, and a neighborhood $U$ of $(x_0,y_0)$ such that for all $(x, y)\in U  \cap A$, the following holds
    \begin{align}
        \det(D^{2}_{\Sigma} u(x)-\langle \textrm{I\!I}_{\Sigma}(x), y \rangle) \leq(1-\varepsilon)\left(\frac{P_\Phi(\Sigma)}{n|\Sigma|}\right)^n.
    \end{align}
    Using Lemma \ref{det-T}, we deduce that on $A$,
    \begin{align}\label{eq: DT < (1-epsilon)C}
        0 \leq \det DT(x, y) \leq (1-\varepsilon1_{U}(x, y))\left(\frac{P_\Phi(\Sigma)}{n|\Sigma|}\right)^n.
    \end{align}
    Certainly in the case where $(x_0, y_0)$ not in $A$, there exists a neighborhood $U$ of $(x_0, y_0)$ disjoint from $A$, and \eqref{eq: DT < (1-epsilon)C} still holds. The inequality \eqref{eq: inequality1} becomes
    \begin{align}\label{sequence-inequ-1}
        \int_{\mathcal{W}} f_\sigma(\xi) d\xi &\leq \int_{\Sigma}\int_{N_x\Sigma}f_\sigma(T(x, y))|\det DT(x, y)|1_{A}(x, y)dy dx\notag\\
        &\leq\left(\frac{P_\Phi(\Sigma)}{n|\Sigma|}\right)^n \int_{\Sigma}\int_{N_x\Sigma}f_\sigma(\nabla^{\Sigma}u(x)+y) dy dx\notag\\
        &-\varepsilon\left(\frac{P_\Phi(\Sigma)}{n|\Sigma|}\right)^n \int_U f_\sigma(\nabla^{\Sigma}u(x)+y).
    \end{align}

    For all $\sigma$ sufficiently close to 1, the integral $\int_U f_{\sigma}(\nabla^{\Sigma}u(x)+y)$ is bounded below by some positive constant independent on $\sigma$. Indeed, equality in \eqref{eq: inequality1} holds when taking the limit on both sides with respect to $f_{\sigma}$, which forces $\int_{N_x\Sigma}f(\nabla^{\Sigma}u(x)+y)dy$ to converge to 1 for almost every $x \in \Sigma$. Thus, the integral is bounded below by the fact that the support of $f_{\sigma}$ concentrates around $S$ by construction. Now taking the limit as $\sigma\to 1$ on both side of \eqref{sequence-inequ-1}, we obtain
    \begin{equation}
        \sup_{\mathcal{F}_{n,m}}\int_{\mathcal{W}} f < \left(\frac{P_\Phi(\Sigma)}{n|\Sigma|^{\frac{n-1}{n}}}\right)^n = |W^*|,
    \end{equation}
     where the last equality follows from Remark \ref{remark: projection=restriction wulff} and the assumption that $\Sigma$ achieves equality of \eqref{wulff-isop}. This leads to a contradiction with the equality of \eqref{opt-inequ}.
\end{proof}

From Claim \ref{ellip-partial}, we obtain
\begin{align}
    D_\Sigma^2 u(x)-\langle \textrm{I\!I}_{\Sigma}(x), y \rangle = \frac{P_{\Phi}(\Sigma)}{n|\Sigma|}\textrm{Id}_n
\end{align}
for all $(x, y) \in N\Sigma$ satisfying $T(x,y)\in S$. Note that the intersection $T(N_x\Sigma)\cap S$ is a $1$-dimensional ellipse in the $2$-dimensional plane $T(N_x\Sigma)$. Indeed this is clear when $\mathcal{W}$ is an ellipsoid; and when $\mathcal{W}$ is some other long convex body, any $x$ such that $T(N_x\Sigma) \cap S$ is not a full ellipse would result in $\int_{N_x\Sigma}f(\nabla^{\Sigma}u(x)+y)dy$ being  strictly smaller than $1$. In this case, equality cannot be achieved when taking limit on both sides of \eqref{eq: inequality1}, so such $x$ forms at most a measure-zero set. Thus, we conclude from above that $ \textrm{I\!I}_{\Sigma}(x)\perp N_x\Sigma$, which implies $ \textrm{I\!I}_{\Sigma}(x) = 0$ and $D_\Sigma^2 u(x) = \frac{P_{\Phi}(\Sigma)}{n|\Sigma|}\textrm{Id}_n$ for all $x\in \Sigma$. Arguing as codimension $m = 1$ case, we conclude that $\Sigma$ is homothetic to some $W^*$. This completes the proof of Proposition \ref{equality}.

\end{proof}

\begin{proof}[Proof of Theorem \ref{thm: sharp cases}]
    Theorem \ref{thm: sharp cases} directly follows from Theorem \ref{aniso-peri}, the fact that $n|W^*| = P_{\Phi}(W^*)$, Proposition \ref{long-body} and Proposition \ref{equality}.
\end{proof}

\section{Proof of Theorem \ref{thm: nonsharp codim 1}}
In  this section, we prove
Theorem \ref{thm: nonsharp codim 1} and we use the following lemma about John's ellipsoid.
\begin{lemma}[John's ellipsoid \cite{Schneider2014}]\label{John} If $K \subset \mathbb{R}^d$ is a centrally symmetric convex body with interior points, then there exists an ellipsoid $E$ such that 
    \begin{equation}
        E \subset K \subset \sqrt{d}E.
    \end{equation}
\end{lemma}

\begin{proof}[Proof of Theorem \ref{thm: nonsharp codim 1}]
Since $\mathcal{W}$ is centrally symmetric, by Lemma \ref{John} there is John's ellipsoid $E \subset \mathbb{R}^{n+m}$ such that
\begin{equation}
    E \subset \mathcal{W} \subset \sqrt{n+m}E.
\end{equation}
Denote $E^*$ as the projection that achieves minimal projection area for $E$, then
\begin{equation}\label{area comparison}
    |E^*| \leq |W^*| \leq (n+m)^{\frac{n}{2}}|E^*|.
\end{equation}
Let $\mathcal{F}_{n,m}^E$ and $\mathcal{F}_{n, m}^{\mathcal{W}}$ be the corresponding sets of functions from \eqref{F_n,m space} for $E$ and $\mathcal{W}$, respectively. We note that $E \subset \mathcal{W}$ implies $\mathcal{F}_{n,m}^E \subset \mathcal{F}_{n, m}^{\mathcal{W}}$.

Define $\tilde{c}_{n,m}$ to be
\begin{equation}
    \tilde{c}_{n,m}=
    \begin{cases}
        1 & m = 1,2 ,\\
        \left(\frac{(n+m)|B^{n+m}|}{m|B^m||B^n|}\right)^{\frac{1}{n}} & m \geq 3.
    \end{cases}
\end{equation}
By Proposition \ref{codim-1} and Proposition \ref{codim-2} we have
\begin{equation}
    \sup_{f\in\mathcal{F}_{n, m}^E}\int_{\mathcal{W}}f \geq \tilde{c}_{n,m}^n|E^*|.
\end{equation}
Therefore
\begin{equation}\label{lower bound f}
    \sup_{f\in\mathcal{F}_{n, m}^{\mathcal{W}}}\int_{\mathcal{W}}f  \geq \sup_{f\in\mathcal{F}_{n, m}^E}\int_{\mathcal{W}}f \geq \tilde{c}_{n,m}^n|E^*| \geq (\frac{1}{\sqrt{n+m}}\tilde{c}_{n,m})^n|W^*| = c_{n,m}^n|W^*|,
\end{equation}
where the third inequality follows from \eqref{area comparison}. Then, by Theorem \ref{aniso-peri} and the fact that $n|W^*| = P_{\Phi}(W^*)$ from \eqref{Wulff-const}, we conclude that
\begin{equation}
    \frac{P_{\Phi}({\Sigma})}{|\Sigma|^{\frac{n-1}{n}}}\geq n(\sup_{\mathcal{F}^{\mathcal{W}}_{n,m}}\int_{\mathcal{W}} f)^{\frac{1}{n}} \geq c_{n, m} \cdot n|W^*|^{\frac{1}{n}} = c_{n, m}\frac{P_{\Phi}({W^*})}{|W^*|^{\frac{n-1}{n}}}.
\end{equation}

\end{proof}

\vspace{5mm}
\bibliography{ref}

\bibliographystyle{alpha}

\end{document}